\documentclass[12pt,a4paper,reqno]{amsart}
\usepackage[latin1]{inputenc}
\usepackage[english]{babel}
\usepackage{amsmath, amsthm, amsfonts,amssymb}

\usepackage[left=2.5cm,right=2.5cm,top=2cm,bottom=2cm]{geometry}

\usepackage{anysize}
\marginsize{10cm}{10cm}{2cm}{2cm}
\newtheorem{thm}{Theorem}[section]
\newtheorem{cor}[thm]{Corollary}
\newtheorem{lem}[thm]{Lemma}
\newtheorem{prop}[thm]{Proposition}
\newtheorem{rem}[thm]{Remark}

\theoremstyle{plain}

\usepackage{hyperref}
\theoremstyle{definition}

\newtheorem{defn}[thm]{Definition}
\theoremstyle{remark}

\newcommand\blfootnote[1]{%
	\begingroup
	\renewcommand\thefootnote{}\footnote{#1}%
	\addtocounter{footnote}{-1}%
	\endgroup
}

\def\ba{\begin{eqnarray*}}
\def\ea{\end{eqnarray*}}

\def\bee{\begin{equation}}

\def\ene{\end{equation}}

\title{A functional characterization of almost-greedy and partially-bases in Banach spaces}
\date{}
\begin{document}
	
			\author{P. M. Bern\'a}
		\address{Pablo M. Bern\'a
			\\
			Departamento de Matem\'atica Aplicada y Estad\'istica, Facultad de Ciencias Econ\'omicas y Empresariales, Universidad San Pablo-CEU, CEU Universities\\ Madrid, 28003 Spain.} \email{pablo.bernalarrosa@ceu.es}
		
			\author{D. Mond\'ejar}
		\address{Diego Mond\'ejar
			\\
			Departamento de Matem\'atica Aplicada y Estad\'istica, Facultad de Ciencias Econ\'omicas y Empresariales, Universidad San Pablo-CEU, CEU Universities\\ Madrid, 28003 Spain.} \email{diego.mondejarruiz@ceu.es}
		\maketitle

		\begin{abstract} 
		In 2003, S. J. Dilworth et al. (\cite{DKKT}) introduced the notion of almost-greedy (resp. partially-greedy) bases. These bases were characterized in terms of quasi-greediness and democracy (resp. conservativeness). In this paper we will show a new functional characterization of these type of bases in general Banach spaces following the spirit of the characterization of greediness proved in \cite{BB}.
			
	\end{abstract}
\blfootnote{\hspace{-0.031\textwidth} 2000 Mathematics Subject Classification. 46B15, 41A65.\newline
		The first author was supported by Grants PID2019-105599GB-I00 (Agencia Estatal de Investigación, Spain) and 20906/PI/18 from Fundaci\'on S\'eneca (Regi\'on de Murcia, Spain). The second author was supported by the research project PGC2018-098321-B-I00 (Ministerio de Ciencia, Innovación y Universidades, Spain). }


	\section{Introduction and background}
	
	Assume that $(\mathbb X,\Vert\cdot\Vert)$ is a Banach space over the field $\mathbb F=\mathbb R$ or $\mathbb C$. Throughout the paper, we assume that $\mathcal B=(e_n)_{n=1}^\infty$ is a semi-normalized Markushevich basis, that is, there exists a unique sequence $(e_n^*)_{n=1}^\infty\subset\mathbb X^*$ such that
	\begin{itemize}
		\item $\overline{\text{span}(e_n\, :\, n\in\mathbb N)}=\mathbb X$,
		\item $e_n^*(e_m)=\delta_{n,m}$,
		\item if $e_n^*(f)=0$ for all $n\in\mathbb N$, then $f=0$,
		\item there are $c_1,c_2>0$ such that
		$$0<c_1:=\inf_n\min \lbrace\Vert e_n\Vert,\Vert e_n^*\Vert\rbrace\leq \sup_n\max\lbrace\Vert e_n\Vert,\Vert e_n^*\Vert\rbrace=:c_2<\infty.$$
	\end{itemize} 
	Hereinafter, by a \textit{basis} for $\mathbb X$, we mean a semi-normalized Markushevich basis. Under these conditions, for each $f\in\mathbb X$, we have that $f\sim \sum_{n=1}^\infty e_n^*(f)e_n$ where $(e_n^*(f))_n\in c_0$. The support of $f\in\mathbb X$ is denoted by $\text{supp}(f)$, where $\text{supp}(f)=\lbrace n\in \mathbb N : \vert e_n^*(f)\vert\neq 0\rbrace$. Finally, we will use the following notation: $\mathbb X_{fin}$ is the subspace of $\mathbb X$ with the elements with finite support, if $f,g\in\mathbb X$, $f\cdot g=0$ means that $\text{supp}(f)\cap \text{supp}(g)=\emptyset$, $\tilde{f}=(e_n^*(f))_{n\in\mathbb N}$ and $\Vert \tilde{f}\Vert_\infty = \sup_{n\in\mathbb N}\vert e_n^*(f)\vert$. Moreover, if $A$ and $B$ are finite sets of natural numbers, $A<B$ means that $\max_{n\in A} n<\min_{j\in B} j$, $P_A$ is the projection operator, that is, $P_A(f)=\sum_{n\in A}e_n^*(f)e_n,$ and $S_k$ is the partial sum of order $k$, that is, $S_k(f)=P_{\lbrace 1,\cdots,k\rbrace}(f).$
	
	In 1999, S. V. Konyagin and V. N. Temlyakov (\cite{VT}) introduced one of the most studied algorithms in the field of non-linear approximation, the so called Thresholding Greedy Algorithm: for $f\in\mathbb X$ and $m\in\mathbb N$, we define a greedy sum of order $m$ as
	$$\mathcal G_m(f)[\mathbb X,\mathcal B]=\mathcal G_m(f):=\sum_{n\in A_m(f)} e_n^*(f)e_n,$$
	where $A_m(f)$ is a greedy set of order $m$, that is, $\vert A_m(f)\vert=m$ and
	$$\min_{n\in A_m(f)}\vert e_n^*(f)\vert\geq \max_{n\not\in A_m(f)}\vert e_n^*(f)\vert.$$
	
	The collection $(\mathcal G_m)_{m\in\mathbb N}$ is the Thresholding Greedy Algorithm. As for every algorithm, one of the first question that we can ask to the audience is when the algorithm converges. To solve that question, S. V. Konyagin and V. N. Temlyakov introduced in \cite{VT} the notion of quasi-greediness.
	
	\begin{defn}
		We say that $\mathcal B$ is quasi-greedy if there is a positive constant $C$ such that
		\begin{eqnarray}\label{quasi}
			\Vert f- \mathcal G_m(f)\Vert\leq C\Vert f\Vert,\;  \forall m\in\mathbb N, \forall f\in\mathbb X.
		\end{eqnarray}
		The least constant verifying \eqref{quasi} is denoted by $C_q=C_q[\mathbb X,\mathcal B]$ and we say that $\mathcal B$ is $C_q$-quasi-greedy.
	\end{defn}
	
	Although this definition only talks about the boundedness of the greedy sums, P. Wojtaszczyk showed in \cite{Woj} that quasi-greediness is equivalent to the convergence of the algorithm. 
	\begin{thm}[\cite{AABW,Woj}]
		A basis $\mathcal B$ in a Banach (or quasi-Banach) space is quasi-greedy if and only if
		\begin{eqnarray*}
			\lim_{m\to+\infty}\Vert f-\mathcal G_m(f)\Vert=0,\; \forall f\in\mathbb X.
		\end{eqnarray*}
	\end{thm}
	
	Then, quasi-greediness is the minimal condition in the convergence of the algorithm, but we are interested in others types of convergence. For instance, when does the algorithm produce the best possible approximation? To study this question, in \cite{VT}, the authors introduced the notion of greediness: a basis $\mathcal B$ is greedy if there is a positive constant $C_g$ such that
	$$\Vert f-\mathcal G_m(f)\Vert\leq C_g\inf\lbrace \Vert f-\sum_{n\in B}a_ne_n\Vert: a_n\in\mathbb F, \vert B\vert\leq m\rbrace,\;\; \forall m\in\mathbb N, \forall f\in\mathbb X.$$	
	
	There are several examples of greedy bases, for instance the canonical basis in the spaces $\ell_p$ with $1\leq p<\infty$, the Haar system in $L_p((0,1))$ with $1<p<\infty$ or the trigonometric system in $L_2(\mathbb T)$. To study greedy bases, S. V. Konyagin and V. N. Temlyakov gave a characterization in terms of unconditional and democratic bases, where a basis is unconditional if the projection operator is uniformly bounded, that is, there is $K>0$ such that, for any finite set $A$,
	$$\Vert P_A(f)\Vert\leq K\Vert f\Vert,\; \forall f\in\mathbb X.$$
	
	Consider $A$ a finite set and define the set of the collection of signs in $A$, $\mathcal E_A=\lbrace \varepsilon=(\varepsilon_n)_{n\in A}: \vert\varepsilon_n\vert=1\rbrace,$ and take the indicator sum
	$$1_{\varepsilon A}=1_{\varepsilon A}[\mathbb X,\mathcal B]:=\sum_{n\in A}\varepsilon_n e_n.$$
	If $\varepsilon\equiv 1$, we will use the notation $1_A$.
	\begin{defn}{\cite{AABW,BBL,DKKT,DKOSS}}
		We say that $\mathcal B$ is symmetric for largest coefficients if there is a positive constant $C$ such that
		\begin{eqnarray}\label{sym}
			\Vert f+1_{\varepsilon A}\Vert\leq C\Vert f+1_{\varepsilon' B}\Vert,
		\end{eqnarray}
		for any pair of sets $\vert A\vert\leq\vert B\vert<\infty$, $A\cap B=\emptyset$, for any $f\in\mathbb X$ such that $\text{supp}(f)\cap (A\cup B)=\emptyset$, $\vert e_n^*(f)\vert\leq 1$ for all $n\in\mathbb N$ and for any choice of signs $\varepsilon\in\mathcal E_A$, $\varepsilon'\in\mathcal E_B$.
		The least constant verifying \eqref{sym} is denoted by $\Delta=\Delta[\mathbb X,\mathcal B]$ and we say that $\mathcal B$ is $\Delta$-symmetric for largest coefficients.
		If \eqref{sym} is satisfied with the extra condition that $A<\text{supp}(f)\cup B$, then we say that $\mathcal B$ is partially symmetric for largest coefficients with constant $\Delta_{pc}$.
	\end{defn}
	
	\begin{defn}
		We say that $\mathcal B$ is super-democratic if there is a positive constant $C$ such that
		\begin{eqnarray}\label{super}
			\Vert 1_{\varepsilon A}\Vert \leq C\Vert 1_{\eta B}\Vert,
		\end{eqnarray}
		for any pair of sets $A, B\subset \mathbb N$, $\vert A\vert\leq\vert B\vert<\infty$ and for any choice of signs $\varepsilon\in\mathcal E_A, \eta\in\mathcal E_B$. The least constant verifying \eqref{super} is denoted by $\Delta_s=\Delta_s[\mathbb X,\mathcal B]$ and we say that $\mathcal B$ is $\Delta_s$-super-democratic.
		
		If \eqref{super} is satisfied for  $A<B$, we say that $\mathcal B$ is $\Delta_{sc}$-super-conservative.
		
		If \eqref{super} is satisfied for  $\varepsilon\equiv\eta\equiv 1$, we say that $\mathcal B$ is $\Delta_d$-democratic and if in addition, $A<B$, we say that $\mathcal B$ is $\Delta_c$-conservative.
		
	\end{defn}

	With these definitions, we can find the following characterizations of greedy bases.
	\begin{thm}
		Assume that $\mathcal B$ is a basis in a Banach space $\mathbb X$.
		\begin{itemize}
			\item $\mathcal B$ is greedy if and only if $\mathcal B$ is democratic and unconditional (see \cite{VT}). Moreover,
			$$\max\lbrace K,\Delta_d\rbrace\leq C_g\leq K+K^2\Delta_d.$$
			\item $\mathcal B$ is greedy if and only if $\mathcal B$ is super-democratic and unconditional (see \cite{BBG}). Moreover,
			$$\max\lbrace K,\Delta_s\rbrace\leq C_g\leq K+K\Delta_s.$$
			\item $\mathcal B$ is greedy if and only if $\mathcal B$ is symmetric for largest coefficients and unconditional (see \cite{DKOSS}). Moreover,
			$$\max\lbrace K,\Delta_d\rbrace\leq C_g\leq K\Delta.$$
		\end{itemize}
	\end{thm}
	
	The last two characterizations were studied with the objective to improve the boundedness constant of greedy bases. Moreover, all the characterizations were given under the assumption of unconditionality and one of the democracy-like properties but, in \cite{BB}, we find a new and interesting property that is so useful to give a new characterization of greediness (see \cite[Corollary 1.8]{BB}). This property is the so called Property (Q): there is a $C>0$ such that
	$$\Vert f+1_{A}\Vert\leq C \Vert f+g+1_{B}\Vert,$$
	for any $\vert A\vert=\vert B\vert<\infty$, $A\cap B=\emptyset$ and $f,g\in\mathbb X_{fin}$ such that $\text{supp}(f)\cap \text{supp}(g)=\emptyset$, $\Vert\tilde{f}\Vert_\infty\leq 1$ and $\text{supp}(f+g)\cap (A\cup B)=\emptyset$.
	
	In that paper, we focus our attention in a closed inequality to characterize the so called almost-greedy and partially-greedy bases.
	
	\begin{defn}[\cite{DKKT}]
		We say that $\mathcal B$ is almost-greedy if there is a positive constant $C$ such that
		\begin{eqnarray}\label{alm}
			\Vert f-\mathcal G_m(f)\Vert\leq C\inf\lbrace \Vert f-P_B(f)\Vert:  \vert B\vert\leq m\rbrace,\;\; \forall m\in\mathbb N, \forall f\in\mathbb X.
		\end{eqnarray}	
		The least constant verifying \eqref{alm} is denoted by $C_{al}=C_{al}[\mathbb X,\mathcal B]$ and we say that $\mathcal B$ is $C_{al}$-almost-greedy.
	\end{defn}
	
	\begin{defn}[\cite{BBL,DKKT}]
		We say that $\mathcal B$ is partially-greedy if there is positive constant $C$ such that
		\begin{eqnarray}\label{partially}
			\Vert f-\mathcal G_m(f)\Vert \leq C\inf_{k\leq m}\Vert f-S_k(f)\Vert,\; \forall m\in\mathbb N, \forall f\in\mathbb X.
		\end{eqnarray}
		The least constant verifying \eqref{partially} is denoted by $C_p=C_p[\mathbb X,\mathcal B]$ and we say that $\mathcal B$ is $C_p$-partially-greedy.
	\end{defn}
	
	\begin{rem}
		In \cite{DKKT}, the condition of partially-greediness was introduced as follows:
		\begin{eqnarray}\label{pa}
			\Vert f-\mathcal G_m(f)\Vert \leq C\Vert f-S_m(f)\Vert,\; \forall m\in\mathbb N, \forall f\in\mathbb X.
		\end{eqnarray}
		Under the condition of Schauder bases, \eqref{pa} and \eqref{partially} are equivalent notions and in \cite{BBL}, the authors proved that if \eqref{pa} is satisfied with $C=1$, then the basis is partially-greedy.
	\end{rem}
	Of course, every greedy basis is almost-greedy and every almost-greedy basis is partially-greedy. One example of an almost-greedy basis that is not greedy is the Lindestrauss basis in $\ell_1$ (\cite{GHO}). Recently, one basis that is partially-greedy and not almost-greedy is presented in \cite[Proposition 6.10]{BDKOW}.  
	
	It is well known that a basis is almost-greedy if and only if the basis is quasi-greedy and democratic and a basis is partially-greedy if the basis is quasi-greedy and conservative (\cite{DKKT,BDKOW}). Moreover, as for greedy bases, we have the following characterizations.
	\begin{thm}
		Assume that $\mathcal B$ is a basis in a Banach space. 
		\begin{itemize}
			\item $\mathcal B$ is almost-greedy if and only if $\mathcal B$ is democratic and quasi-greedy (\cite{DKKT}). Moreover,
			$$\max\lbrace C_q,\Delta_d\rbrace\leq C_{al}\leq 8C_q^4\Delta_d+C_q+1.$$
			\item$\mathcal B$ is almost-greedy if and only if $\mathcal B$ is super-democratic and quasi-greedy (\cite{BBG}). Moreover,
			$$\max\lbrace C_q,\Delta_s\rbrace\leq C_{al}\leq C_q+C_q\Delta_s.$$
			\item $\mathcal B$ is almost-greedy if and only if $\mathcal B$ is symmetric for largest coefficients and quasi-greedy (\cite{BBG}). Moreover,
			$$\max\lbrace C_q,\Delta\rbrace\leq C_{al}\leq C_q\Delta.$$
		\end{itemize}
	\end{thm}
	
	\begin{thm}{\cite{BBL,DKKT}}\label{mp}
		Assume that $\mathcal B$ is a basis in a Banach space. 
		\begin{itemize}
			\item $\mathcal B$ is partially-greedy if and only if $\mathcal B$ is conservative and quasi-greedy. Moreover,
			$$\max\lbrace C_q,\Delta_c\rbrace\leq C_{p}\leq C_q+C_q^2(1+C_q)\Delta_c.$$
			\item $\mathcal B$ is partially-greedy if and only if $\mathcal B$ is super-conservative and quasi-greedy. Moreover,
			$$\max\lbrace C_q,\Delta_{sc}\rbrace\leq C_{p}\leq C_q+C_q(1+C_q)\Delta_{sc}.$$
			\item $\mathcal B$ is partially-greedy if and only if $\mathcal B$ is partially-symmetric for largest coefficients and quasi-greedy. Moreover,
			$$\max\lbrace C_q,\Delta_{pc}\rbrace\leq C_{p}\leq C_q\Delta_{pc}.$$
		\end{itemize}
	\end{thm}
	
	The purpose of this paper is to get a new characterization of almost-greedy and partially-greedy bases following the ideas of \cite[Corollary 1.8]{BB} for greedy bases.
	
	\begin{defn}
		We say that $\mathcal B$ has the Property (F) if there is a positive constant $C$ such that
		\begin{eqnarray}\label{efe}
			\Vert f+1_{A}\Vert\leq C \Vert f+g+1_{B}\Vert,
		\end{eqnarray}
		for any $A, B, f, g$ satisfying the following conditions:
		\begin{itemize}
			\item[i)] $\vert A\vert\leq\vert B\vert<\infty$ and $A\cap B=\emptyset$, 
			\item[ii)] $f,g\in\mathbb X_{fin}$, $f\cdot g=0$, $\text{supp}(f+g)\cap (A\cup B)=\emptyset$, $\Vert\tilde{f}\Vert_\infty\leq 1$ and $\Vert\tilde{f}\Vert_\infty\leq\inf_{n\in\text{supp}(g)}\vert e_n^*(g)\vert$.
		\end{itemize}
		
		The least constant verifying \eqref{efe} is denoted by $\mathcal F=\mathcal F[\mathbb X,\mathcal B]$ and we say that $\mathcal B$ has the Property (F) with constant $\mathcal F$.
		
		Also, if \eqref{efe} is satisfied with the extra condition that $A<\text{supp}(g)\cup B$, we say that $\mathcal B$ has the Property (F$_p$) with constant $\mathcal F_p$.
	\end{defn}

	\begin{defn}
		We say that $\mathcal B$ has the Property (F$^*$) if there is a positive constant $C$ such that
		\begin{eqnarray}\label{efee}
			\Vert f+z\Vert\leq C \Vert f+y\Vert,
		\end{eqnarray}
		for any $f,z,y\in\mathbb X_{fin}$ satisfying the following conditions:
		
		\begin{itemize}
			\item[i)] $f\cdot z=0$, $f\cdot y=0$, $z\cdot y=0$,
			\item[ii)]  $\max\lbrace\Vert\tilde{f}\Vert_\infty,\Vert \tilde{z}\Vert_\infty\rbrace\leq 1$ .
			\item[iii)] $\vert D\vert\geq \vert \text{supp}(z)\vert$, where $D=\lbrace n\in \text{supp}(y) : \vert e_n^*(y)\vert=1\rbrace$.
			\item[iv)] $\inf_{n\in \text{supp}(y)}\vert e_n^*(y)\vert\geq \Vert\tilde{f}\Vert_\infty$.
		\end{itemize}
		The least constant verifying \eqref{efee} is denoted by $\mathcal F^*=\mathcal F^*[\mathbb X,\mathcal B]$ and we say that $\mathcal B$ has the Property (F$^*$) with constant $\mathcal F^*$.
		
		Also, if \eqref{efee} is satisfied with the extra condition that $\text{supp(z)}<\text{supp}(f+y)$, we say that $\mathcal B$ has the Property (F$^*_p$) with constant $\mathcal F_p^*$.
	\end{defn}
	
	The main theorems that we will prove are the following.
	\begin{thm}\label{main}
		Let $\mathcal B$ be a basis in a Banach space $\mathbb X$. 
		\begin{itemize}
			\item[a)] If $\mathcal B$ is almost-greedy with constant $C_{al}$, then $\mathcal B$ has the Property (F$^*$) with constant $\mathcal F^*\leq C_{al}(1+2C_{al})$.
			\item[b)]If $\mathcal B$ has the Property (F$^*$) with constant $\mathcal F^*$, then the basis is almost-greedy with constant $C_{al}\leq (\mathcal F^*)^2$.
		\end{itemize}
	\end{thm}
	
	\begin{thm}\label{maintwo}
		Let $\mathcal B$ be a basis in a Banach space $\mathbb X$. 
		\begin{itemize}
			\item[a)] If $\mathcal B$ is partially-greedy with constant $C_{p}$, then $\mathcal B$ has the Property (F$^*_p$) with constant $\mathcal F^*_p\leq C_p(1+2C_p)$.
			\item[b)]If $\mathcal B$ has the Property (F$^*_p$) with constant $\mathcal F^*_p$, then the basis is partially-greedy with constant $C_{p}\leq (\mathcal F^*_p)^2$.
		\end{itemize}
	\end{thm}
	
	The structure of the paper is the following: in Section \ref{ff}, we will show some basics about the Properties (F) and (F$^*$). In Section \ref{ff1}, we prove Theorem \ref{main}. In Section \ref{ff2} we give a brief summary about the Properties (F$_p$) and (F$_p^*$), in Section \ref{ff3} we prove Theorem \ref{maintwo} and, finally, in Section \ref{ff4}, we give some density results that we use in the paper.
	\section{Properties (F) and (F$^*$)}\label{ff}
	This section is focused in the study of the Properties (F) and (F*). In fact, we will show that these properties are equivalent. To show that we will need some auxiliary lemmas about convexity. 
	\begin{lem}{\cite[Corollary 2.3]{AABW}}\label{convex1}
		Let $\mathbb X$ be a Banach space, let $\mathcal B$ be a basis for $\mathbb X$ and $J$ a finite set.
		\begin{itemize}
			\item[(i)] For any scalars $(a_j)_{j\in J}$ with $0\leq a_j\leq 1$ and any $g\in\mathbb X$,
			$$\Vert g+\sum_{j\in J}a_j e_j\Vert \leq \sup\lbrace \Vert g+1_{A}\Vert\, :\; A\subseteq J\rbrace.$$
			\item[(ii)] For any scalars $(a_j)_{j\in J}$ with $\vert a_j\vert\leq 1$ and any $g\in\mathbb X$,
			$$\Vert g+\sum_{j\in J}a_j e_j\Vert \leq \sup_{\varepsilon\in\mathcal E_J} \Vert g+1_{\varepsilon J}\Vert.$$
		\end{itemize}
		
	\end{lem}

	\begin{lem}\label{guau}
		Let $\mathcal B$ be a basis of a Banach space $\mathbb X$. Then,
		$$\sup_{\varepsilon\in\mathcal E_A}\Vert f+1_{\varepsilon A}\Vert \leq 5\sup_{B\subseteq A}\Vert f+1_B\Vert.$$
	\end{lem}
	\begin{proof}
		If $\mathbb F=\mathbb R$, following the result \cite[Lemma 2.3]{BB}, we know that
		\begin{eqnarray*}\label{real}
			\sup_{\varepsilon_n=\pm 1}\Vert f+1_{\varepsilon A}\Vert \leq 3\sup_{B\subseteq A}\Vert f+1_B\Vert.
		\end{eqnarray*}
		We prove now the result for the complex case. In that case,
		\begin{eqnarray}\label{decom}
			\nonumber	1_{\varepsilon A}&=&\sum_{n\in A}\text{Re}(\varepsilon_n)e_n+i\sum_{n\in A}\text{Im}(\varepsilon_n)e_n\\
			\nonumber	&=&\sum_{n\in A_1}\text{Re}^+(\varepsilon_n)e_n-\sum_{n\in A_2}\text{Re}^-(\varepsilon_n)e_n\\
			&+&i\left(\sum_{n\in A_3}\text{Im}^+(\varepsilon_n)e_n-\sum_{n\in A_4}\text{Im}^-(\varepsilon_n)e_n\right),
		\end{eqnarray}
		where $A_i$ are the corresponding subsets of $A$. Then,
		\begin{eqnarray*}
			\Vert f+1_{\varepsilon A}\Vert&\leq& \Vert f+\sum_{n\in A}\text{Re}(\varepsilon_n)e_n\Vert + \Vert \sum_{n\in A}\text{Im}(\varepsilon_n)e_n\Vert\\	
			&\leq & \Vert f\Vert + \Vert f+\sum_{n\in A_1}\text{Re}^+(\varepsilon_n)e_n\Vert + \Vert f+\sum_{n\in A_2}\text{Re}^-(\varepsilon_n)e_n\Vert\\
			&+& \Vert f+\sum_{n\in A_3}\text{Im}^+(\varepsilon_n)e_n\Vert + \Vert f+\sum_{n\in A_4}\text{Im}^-(\varepsilon_n)e_n\Vert\\
			&\underset{\text{Lemma}\, \ref{convex1} }{\leq}& 5\sup_{B\subseteq A}\Vert f+1_B\Vert.
		\end{eqnarray*}
	\end{proof}
	
	\begin{thm}\label{p1}
		Let $\mathcal B$ be a basis in a Banach space $\mathbb X$. The basis is democratic (or symmetric for largest coefficients) and quasi-greedy if and only if the basis has the Property (F). Concretely:
		\begin{enumerate}
			\item[(1)] If $\mathcal B$ has the Property (F) with constant $\mathcal F$, then the basis is $C_q$-quasi-greedy and $\Delta_d$-democratic with
			$$\max\lbrace C_q, \Delta_d\rbrace\leq \mathcal F.$$
			\item[(2)] If $\mathcal B$ has the Property (F) with constant $\mathcal F$, then the basis is $C_q$-quasi-greedy and $\Delta$-symmetric for largest coefficients with
			$$C_q\leq \mathcal F,\;\; \Delta\leq 5(\mathcal F+4\mathcal F^2+4\mathcal F^3).$$
			
			\item[(3)] If $\mathcal B$ is $\Delta_d$-democratic and $C_q$-quasi-greedy, then the basis has the Property (F) with constant 
			$$\mathcal F\leq C_q(1+(1+C_q)\Delta_d).$$
			\item[(4)] If $\mathcal B$ is $\Delta$-symmetric for largest coefficients and $C_q$-quasi-greedy, then the basis has the Property (F) with constant 
			$$\mathcal F\leq 3\Delta C_q.$$
		\end{enumerate}
	\end{thm}
	
	\begin{proof}
		First of all, we show (1). Assume that the basis has the Property (F) with constant $\mathcal F$. To show that $\mathcal B$ is quasi-greedy, we take $f\in\mathbb X_{fin}$ with $t=\Vert\tilde{f}\Vert_\infty$ and $m\in\mathbb N$. Then, if we take in the definition of the Property (F) $f'=\frac{f}{t}-\mathcal G_m(\frac{f}{t})$, $g'=\mathcal G_m(\frac{f}{t})$ and $A=B=\emptyset$, since $\Vert \tilde{f'}\Vert_\infty\leq \inf_{n\in\text{supp}(g')}\vert e_n^*(g')\vert$, we obtain that
		$$\Vert f-\mathcal G_m(f)\Vert=t\Vert f'\Vert \leq \mathcal F\,t\Vert f'+g'\Vert=\mathcal F\Vert f\Vert,$$
		so the basis is quasi-greedy with $C_q\leq \mathcal F$ for elements with finite support. To obtain that $\mathcal B$ is quasi-greedy for any $f\in\mathbb X$, we use Corollary \ref{bb2}.
		
		Prove now that the basis is democratic. For that, we take $C$ and $D$ two finite sets such that $\vert C\vert\leq \vert D\vert$. Now, we do the following decomposition:
		$$D=(D\cap C) \cup D_1 \cup D_2,$$
		where $\vert D_1\vert=\vert C\setminus D\vert$ and $D_1\cap D_2=\emptyset$. Hence, taking $f=1_{D\cap C}$, $g=1_{D_2}$, $A=C\setminus D$ and $B=D_1$,
		$$\Vert 1_C\Vert=\Vert 1_{C\cap D}+1_{C\setminus D}\Vert\leq \mathcal F\Vert 1_{D\cap C}+1_{D_2}+1_{D_1}\Vert=\mathcal F\Vert 1_D\Vert.$$
		Thus, $\mathcal B$ is democratic with $\Delta_d\leq \mathcal F$.
		
		Prove now (2). We only have to show that $\mathcal B$ is symmetric for largest coefficients. For that, take $f\in\mathbb X_{fin}$, $\Vert\tilde{f}\Vert_\infty\leq 1$, $A\cap B=\emptyset$, $\vert A\vert\leq\vert B\vert<\infty$, $\text{supp}(f)\cap (A\cup B)=\emptyset$, $\varepsilon\in\mathcal E_A$ and $\eta\in\mathcal E_B$. Using Lemmas \ref{guau} and \ref{convex1}, we only have to show that there is some absolute constant $C$ such that
		$$\Vert f+1_{A'}\Vert \leq C \Vert f+1_{\eta B}\Vert,\; \forall A'\subseteq A.$$

		Of course, since the Property (F) implies quasi-greediness with constant $C_q\leq\mathcal F$ by (1), if we take the element $h:=f+1_{\eta B}$ with $\Vert \tilde{f}\Vert_\infty \leq 1$, we have
		\begin{eqnarray}\label{aux}
			\Vert f\Vert = \Vert h-\mathcal G_{\vert B\vert}(h)\Vert \leq \mathcal F\Vert h\Vert=\mathcal F\Vert f+1_{\eta B}\Vert.
		\end{eqnarray}
		
		Also, respect to the set $A'$, we can have the following:
		\begin{eqnarray}\label{ap}
			\Vert 1_{A'}\Vert \leq \mathcal F\Vert 1_B\Vert \leq 4\mathcal F^2\Vert 1_{\eta B}\Vert,
		\end{eqnarray}
		where in the last inequality we have used \cite[Proposition 2.1.11]{B} or \cite[Lemma 3.2]{AABW}.\footnote{These results affirm that for quasi-greedy bases, $\Vert 1_{\varepsilon A}\Vert\leq 2\kappa C_q \Vert 1_{\eta A}\Vert$, for any $\eta,\varepsilon\in\mathcal E_A$ and any finite set $A$ with $\kappa=1$ if $\mathbb F=\mathbb R$ and $\kappa=2$ if $\mathbb F=\mathbb C$.}

		Thus, 
		\begin{eqnarray*}
			\Vert f+1_{A'}\Vert&\leq&\Vert f\Vert+\Vert 1_{A'}\Vert\underset{\eqref{aux}+\eqref{ap}}{\leq} \mathcal F\Vert f+1_{\eta B}\Vert+4\mathcal F^2\Vert 1_{\eta B}\Vert\\
			&\leq& (\mathcal F+4\mathcal F^2)\Vert f+1_{\eta B}\Vert + 4\mathcal F^2\Vert f\Vert\\
			&\leq&(\mathcal F+4\mathcal F^2+4\mathcal F^3)\Vert f+1_{\eta B}\Vert.
		\end{eqnarray*}
		
		Finally, applying convexity,
		
		$$\Vert f+1_{\varepsilon A}\Vert  \underset{\text{Lemma}\,\ref{guau}}{\leq} 5\sup_{A'\subseteq A}\Vert f+1_{A'}\Vert \leq 5(\mathcal F+4\mathcal F^2+4\mathcal F^3)\Vert f+1_{\eta B}\Vert.$$
		
		So, the basis is symmetry for largest coefficients for elements with finite support with constant
		$$\Delta\leq 5(\mathcal F+4\mathcal F^2+4\mathcal F^3).$$
		
		Applying Lemma \ref{sy}, the result follows for any $f\in\mathbb X$.

		(3) Assume now that $\mathcal B$ is $C_q$-quasi-greedy and $\Delta_d$-democratic and take $f, g\in\mathbb X_{fin}$ with $f\cdot g=0$, $\inf_{n\in \text{supp}(g)}\vert e_n^*(g)\vert\geq\Vert\tilde{f}\Vert_\infty$, $A\cap B=\emptyset$, $\vert A\vert\leq \vert B\vert<\infty$ and $\text{supp}(f+g)\cap (A\cup B)=\emptyset$.
		\begin{eqnarray}\label{one}
			\Vert f+1_A\Vert \leq \Vert f\Vert+\Vert 1_A\Vert\leq \Vert f\Vert+\Delta_d\Vert 1_B\Vert.
		\end{eqnarray}
		If we take $h:=f+g+1_B$, it is clear that $\text{supp}(g+1_B)$ is a greedy set of $h$. Then, if $\vert \text{supp}(g+1_B)\vert=n$,
		\begin{eqnarray}\label{two}
			\Vert f\Vert=\Vert h-\mathcal G_n(h)\Vert \leq C_q\Vert h\Vert=C_q\Vert f+g+1_B\Vert.
		\end{eqnarray}
		Since $\inf_{n\inf_{n\in \text{supp}(g)}}\vert e_n^*(g)\vert\geq\Vert\tilde{f}\Vert_\infty$ and $\Vert\tilde{f}\Vert_\infty\leq 1$, we can decompose $g$ as $g=g_1+g_2$, where $\text{supp}(g_1)=\lbrace n\in\text{supp}(g) : \vert e_n^*(g)\vert\geq 1\rbrace$ and $\text{supp}(g_2)=\lbrace n\in\text{supp}(g) : \vert e_n^*(g)\vert< 1\rbrace$. Then, if we take $u:=f+g_2+1_B$, $B$ is a greedy set for $u$ with of order $k:=\vert B\vert$, and taking $v=u+g_1$, $\text{supp}(g_1)$ is a greedy set of $v$ of order $p:=\vert \text{supp}(g_1)\vert$. Thus,
		\begin{eqnarray}\label{three}
			\nonumber\Vert 1_B\Vert &=&\Vert \mathcal G_k(u)\Vert \leq (1+C_q)\Vert u\Vert=(1+C_q)\Vert f+g_2+1_B\Vert\\
			&=&(1+C_q)\Vert v-\mathcal G_p(v)\Vert\leq (1+C_q)C_q\Vert f+g+1_B\Vert.
		\end{eqnarray}
		Adding up \eqref{two} and \eqref{three} in \eqref{one}, we obtain the result, that is, the basis has the Property (F) with $\mathcal F\leq C_q(1+(1+C_q)\Delta_d)$.

		(4) Finally, assume that $\mathcal B$ is $\Delta$-symmetric for largest coefficients and $C_q$-quasi-greedy. Take $f, g, A$ and $B$ as in the Property (F). Then,
		\begin{eqnarray*}
			\Vert f+1_A\Vert&\leq& \Delta \Vert f+1_B\Vert\leq\Delta (\Vert f+g_1+1_B\Vert+\Vert g_1+f\Vert+\Vert f\Vert)\\
			&\leq& 3C_q\Delta\Vert f+g+1_B\Vert. 
		\end{eqnarray*}
		Thus, the basis has the Property (F) with constant $\mathcal F\leq 3C_q\Delta$.  
	\end{proof}
	
	\begin{thm}\label{p2}
		Let $\mathcal B$ be a basis in a Banach space $\mathbb X$. The basis has the Property (F) if and only if the basis has the Property (F$^*$). Moreover, if $\mathcal F$ and $\mathcal F^*$ are the constants of the corresponding properties, then
		$$\mathcal F\leq \mathcal F^*\leq 5\mathcal F(1+2\mathcal F+8\mathcal F^2).$$
	\end{thm}
	
	\begin{proof}
		Assume that we have the Property (F$^*$) with constant $\mathcal F^*$ and take $f,g,A$ and $B$ as in the Property (F), that is, $f\cdot g=0$, $A\cap B=\emptyset$, $\vert A\vert\leq\vert B\vert$, $\text{supp}(f+g)\cap(A\cup B)=\emptyset$, $\Vert\tilde{f}\Vert_\infty\leq 1$ and $\Vert \tilde{f}\Vert_\infty\leq \inf_{n\in \text{supp}(g)}\vert e_n^*(g)\vert$. 
		Taking $z=1_A$ and $y=g+1_B$ in the Property (F$^*$), $f,z$ and $y$ verify the conditions established  in the Property (F$^*$). Then,
		$$\Vert f+1_A\Vert=\Vert f+z\Vert\leq \mathcal F^*\Vert f+y\Vert=\mathcal F^*\Vert f+g+1_B\Vert,$$
		so the basis has the Property (F) with $\mathcal F\leq \mathcal F^*$.

		Assume now that we have the Property (F) and take $f,y$ and $z$ in $\mathbb X_{fin}$ as in the Property (F$^*$), that is, $f\cdot z=0$, $f\cdot y=0$, $z\cdot y =0$, $\max\lbrace \Vert \tilde{f}\Vert_\infty, \Vert \tilde{z}\Vert_\infty\rbrace\leq 1$ and $\vert\text{supp}(z)\vert\leq \vert D\vert$ where $D=\lbrace n : \vert e_n^*(y)\vert=1\rbrace\vert$. Using now Lemmas \ref{guau} and \ref{convex1}, it is enough to prove that there exists $C_1>0$ such that
		$$\Vert f+1_{A'}\Vert\leq C_1\Vert f+y\Vert,\; \forall A'\subseteq A,$$
		where $A=\text{supp}(z)$. Using Property (F), we have that
		\begin{eqnarray}\label{qg}
			\Vert h\Vert\leq \mathcal F\Vert h+w\Vert,
		\end{eqnarray}
		for any $h$ and $w$ such that $h\cdot w=0$ and $\inf_{n\in \text{supp}(w)}\vert e_n^*(w)\vert\geq \Vert\tilde{h}\Vert_\infty$.
		
		Taking $D=\lbrace n : \vert e_n^*(y)\vert=1\rbrace$, observe that $y=P_{D^c}(y)+1_{\eta D}$, where $\eta\equiv \lbrace \text{sign}(e_n^*(y))\rbrace_n$. Then, if $A'\subseteq A$, 
		\begin{eqnarray}\label{pone}
			\Vert f+1_{A'}\Vert \leq \mathcal F\Vert f+P_{D^c}(y)+1_{D}\Vert\leq \mathcal F\left(\Vert f+y\Vert+\Vert 1_{\eta D}\Vert+\Vert 1_{D}\Vert\right).
		\end{eqnarray}
		
		If we decompose $D^c=D_1\cup D_2$ such that 
		$$D_1=\lbrace n : \vert e_n^*(y)\vert <1\rbrace,\;\; D_2=\lbrace n : \vert e_n^*(y)\vert >1\rbrace,$$
		we obtain that
		\begin{eqnarray}\label{ptwo}
			\Vert 1_{\eta D}\Vert \leq \Vert f+P_{D_1}(y)+1_{\eta D}\Vert + \Vert f+P_{D_1}(y)\Vert\underset{\eqref{qg}}{\leq} 2\mathcal F\Vert f+y\Vert.
		\end{eqnarray}
		Following the idea of \eqref{ap}, we can obtain that
		\begin{eqnarray}\label{pthree}
			\Vert 1_{D}\Vert \leq 4\mathcal F\Vert 1_{\eta D}\Vert \underset{\eqref{ptwo}}{\leq} 8\mathcal F^2\Vert f+y\Vert.
		\end{eqnarray}
		
		Using \eqref{ptwo} and \eqref{pthree} in \eqref{pone}, we have
		$$\Vert f+1_{A'}\Vert \leq \mathcal F(1+2\mathcal F+8\mathcal F^2)\Vert f+y\Vert.$$
		Using now Lemma \ref{convex1}, we obtain 
		$$\Vert f+z\Vert\leq \sup_{\varepsilon\in\mathcal E_A}\Vert f+1_{\varepsilon A}\Vert \underset{\text{Lemma}\,\ref{guau}}{\leq}5\sup_{A'\subseteq A}\Vert f+1_{A'}\Vert\leq 5\mathcal F(1+2\mathcal F+8\mathcal F^2)\Vert f+y\Vert.$$
		So the basis has the Property (F*) with $\mathcal F^*\leq 5\mathcal F(1+2\mathcal F+8\mathcal F^2)$.
	\end{proof}
	
	To finish this section, we give the following nice characterization of the Property (F$^*$) that will be useful to show our main theorem.
	
	\begin{prop}\label{prop1}
		Let $\mathcal B$ be a basis in a Banach space $\mathbb X$. The following are equivalent:
		\begin{itemize}
			\item[i)] There is a positive constant $C$ such that
			\begin{eqnarray}\label{c1}
				\Vert f+1_{\varepsilon A}\Vert \leq C\Vert f+g+1_{\eta B}\Vert,
			\end{eqnarray}
			for any $f, g\in\mathbb X_{fin}$ such that $f\cdot g=0$, $\Vert\tilde{f}\Vert_\infty\leq 1$ and $\inf_{n\in\text{supp(g)}}\vert e_n^*(g)\vert\geq \Vert\tilde{f}\Vert_\infty$, for any pair of finite sets $A$ and $B$ such that $A\cap B=\emptyset$, $\vert A\vert\leq\vert B\vert$, $\text{supp}(f+g)\cap (A\cup B)=\emptyset$, and for any $\varepsilon\in\mathcal E_A$, $\eta\in\mathcal E_B$. 
			\item[ii)] The basis has the Property (F*) with constant $\mathcal F^*$.
			\item[iii)] There is a positive constant $C$ such that
			\begin{eqnarray}\label{c2}
				\Vert f\Vert \leq C\Vert f-P_A(f)+y\Vert,
			\end{eqnarray}
			for any $f, y\in\mathbb X_{fin}$ with $f\cdot y=0$ and $A\subseteq \text{supp}(f)$ verifying 
			
			\begin{itemize}
				\item[a)] $ \Vert\tilde{f}\Vert_\infty\leq 1$.
				\item[b)] $\inf_{n\in \text{supp}(y)}\vert e_n^*(y)\vert\geq \Vert\tilde{f}\Vert_\infty$.
				\item[c)] $\vert D\vert\geq \vert A\vert$, where $D=\lbrace n\in\text{supp}(y): \vert e_n^*(y)\vert=1\rbrace$.
			\end{itemize}
		\end{itemize}
		Moreover, if we denote by $C_1$ and $C_2$ the least constants verifying \eqref{c1} and \eqref{c2} respectively, we have
		$$\mathcal F^*\leq C_1,\; C_2\leq \mathcal F^*,\; C_1\leq C_2.$$
	\end{prop}
	
	\begin{proof}
		First, we prove i) $\Rightarrow$ ii). Take $f,z,y\in\mathbb X_{fin}$ as in the definition of the Property (F$^*$):
		
		\begin{itemize}
			\item $f\cdot y=0$, $f\cdot z=0$, $z\cdot y=0$,
			\item  $\max\lbrace\Vert\tilde{f}\Vert_\infty, \Vert \tilde{z}\Vert_\infty\rbrace\leq 1$ .
			\item $\vert D\vert\geq \vert \text{supp}(z)\vert$, where $D=\lbrace n : \vert e_n^*(y)\vert=1\rbrace$.
			\item $\inf_{n\in \text{supp}(y)}\vert e_n^*(y)\vert\geq \Vert\tilde{f}\Vert_\infty$.
		\end{itemize}
		
		If $z=0$, just take $A=B=\emptyset$ and the prove is over. Consider now that $z\neq 0$ and take $\text{supp}(z)=A$. If we divide $y=1_{\eta D}+P_{D^c}(y)$ with $\eta\equiv\lbrace \text{sign}(e_n^*(y))\rbrace$, we have for all $\varepsilon\in\mathcal E_A$,
		\begin{eqnarray}
			\Vert f+1_{\varepsilon A}\Vert \leq C_1\Vert f+P_{D^c(y)}+1_{\eta D}\Vert=C_1\Vert f+y\Vert
		\end{eqnarray}
		Applying now Lemma \ref{convex1}, we obtain the result with $\mathcal F^*\leq C_1$.
		
		Now, we show that ii) $\Rightarrow$ iii). Of course, if $A=\emptyset$, the result is trivial. Take $f,y$ and $A$ as in  iii) with $A\neq\emptyset$ and $A\subseteq\text{supp}(f)$. If in the Property (F*) we take $f'=f-P_A(f)$, $z'=1_{\varepsilon A}$ with $\varepsilon\in\mathcal E_A$ and $y'=y$,
		$$\Vert f'+z'\Vert=\Vert f-P_A(f)+1_{\varepsilon A}\Vert \leq\mathcal F^*\Vert f'+y'\Vert=\mathcal F^*\Vert f-P_A(f)+y\Vert,$$
		so applying  the item (ii) of Lemma \ref{convex1}, iii) is proved with $C_2\leq \mathcal F^*$.
		
		Finally, we make the proof to show that iii) $\Rightarrow$ i). Take $f,g\in\mathbb X_{fin}$ such that $f\cdot g=0$, $\Vert\tilde{f}\Vert_\infty\leq\inf_{n\in\text{supp}(g)}\vert e_n^*(g)\vert$, $\vert A\vert\leq\vert B\vert<\infty$, $A\cap B=\emptyset$, $\text{supp}(f+g)\cap(A\cup B)=\emptyset$ and $\varepsilon\in\mathcal E_A$, $\eta\in\mathcal E_B$.
		
		Taking $f'=f+1_{\varepsilon A}$ and $y=g+1_{\eta B}$,
		$$\Vert f+1_{\varepsilon A}\Vert =\Vert f'\Vert\leq C_2\Vert f'-P_A(f')+y\Vert= C_2\Vert f+g+1_{\eta B}\Vert,$$
		so the proof is over and $C_1\leq C_2$.
	\end{proof}
	%

	\section{Proof of Theorem \ref{main}}\label{ff1}
	To prove Theorem \ref{main}, we will use one of the most important tools in the world of quasi-greedy bases: the truncation operator. To define this operator, we take $\alpha>0$ and define, first of all, the $\alpha$-truncation of $z\in\mathbb C$:
	$$T_\alpha(z)=\alpha\text{sign}(z),\;\; \text{if}\, \vert z\vert\geq \alpha,$$
	and
	$$T_\alpha(z)=z,\;\;\text{if}\, \vert z\vert\leq \alpha.$$
	
	Now, it is possible to extend $T_\alpha$ to an operator in the space $\mathbb X$ by
	$$T_\alpha(f)=\sum_{n\in\text{supp}(f)}T_\alpha(e_n^*(f))e_n=\sum_{n\in \Delta_\alpha}\alpha\frac{e_n^*(f)}{\vert e_n^*(f)\vert}e_n + \sum_{n\not\in\Delta_\alpha}e_n^*(f)e_n,$$
	where the set $\Delta_\alpha=\lbrace n\in\mathbb N : \vert e_n^*(f)\vert>\alpha\rbrace$. Of course, since $\Delta_\alpha$ is a finite set, $T_\alpha$ is well-defined for all $f\in\mathbb X$.
	\begin{lem}{\cite[Lemma 2.5]{BBG}}\label{trunc}
		Let $\mathcal B$ be a $C_q$-quasi-greedy basis in a Banach space. Then, the truncation operator is uniformly bounded that is,
		$$\Vert T_\alpha(f)\Vert\leq C_q\Vert f\Vert,\; \forall \alpha>0,\; \forall f\in\mathbb X.$$
	\end{lem}
	\begin{proof}[Proof of Theorem \ref{main}]
		Assume that $\mathcal B$ is almost-greedy with constant $C_{al}$ and take $f,z$ and $y$ as in the Property (F$^*$) and decompose $y=P_{B_1}(y)+P_{B_2}(y)+1_{\eta B}$, where $\eta\equiv\lbrace\text{sign}(e_n^*(y))\rbrace$, $B_1\cup B_2=B^c$ and
		$$B_1=\lbrace n : \vert e_n^*(y)\vert<1\rbrace,\;\; B_2=\lbrace n : \vert e_n^*(y)\vert>1\rbrace.$$
		Taking now $h:=f+1_{\varepsilon A}+P_{B_2}(y)+1_{\eta B}$ with $A=\text{supp}(z)$, $\varepsilon\in\mathcal E_A$ and $n=\vert B_2\vert+\vert B\vert$, we obtain
		\begin{eqnarray*}
			\Vert f+1_{\varepsilon A}\Vert &=& \Vert h-\mathcal G_n(h)\Vert \leq C_{al}\Vert h-P_A(h)\Vert=C_{al}\Vert f+P_{B_2}(y)+1_{\eta B}\Vert\\
			&\leq& C_{al}(\Vert f+y\Vert+\Vert P_{B_1}(y)\Vert)\leq C_{al}(\Vert f+y\Vert +\Vert f+P_{B_1}(y)\Vert+\Vert f\Vert)\\
			&\leq& C_{al}(\Vert f+y\Vert + 2C_{al}\Vert f+y\Vert)\\
			&\leq& C_{al}(1+2C_{al})\Vert f+y\Vert.
		\end{eqnarray*}
		Thus, applying Lemma \ref{convex1}, the basis has the Property (F$^*$) with constant $\mathcal F^*\leq C_{al}(1+2C_{al})$.
		
		Assume now that the basis has the Property (F$^*$). Take $f\in\mathbb X_{fin}$, $m\in\mathbb N$, $\mathcal G_m(f)=P_G(f)$ and $\vert A\vert\leq m$.
		
		Consider now the elements $f'=\frac{1}{t}(f-\mathcal G_m(f))$ with $t=\min_{n\in G\setminus A}\vert e_n^*(f)\vert$,  $B=A\setminus G$, $y=1_{\eta(G\setminus A)}$ and  $\eta\equiv\lbrace \text{sign}(e_n^*(f))\rbrace$. Of course, $f'\cdot y=0$, $\Vert\tilde{f'}\Vert_\infty \leq 1$ since $\vert e_n^*(f-\mathcal G_m(f))\vert\leq t$ for $n\in G^c$ and $\vert G\setminus A\vert\geq \vert B\vert$. Then, applying these elements in the item iii) of Proposition \ref{prop1}, we obtain the following:
		\begin{eqnarray}\label{main1}
			\nonumber\Vert f-\mathcal G_m(f)\Vert &=&t\Vert f'\Vert\leq t\mathcal F^*\Vert f'-P_{B}(f')+y\Vert\\
			\nonumber	&=& \mathcal F^*\Vert f-P_G(f)-P_{A\setminus G}(f)+ t1_{\eta(G\setminus A)}\Vert\\
			&=&\mathcal F^*\Vert P_{(A\cup G)^c}(f-P_A(f))+t1_{\eta(G\setminus A)}\Vert.
		\end{eqnarray}
		Since the Property (F$^*$) implies that the basis is quasi-greedy with $C_q\leq\mathcal F^*$ (Theorems \ref{p1} and \ref{p2}), applying Lemma \ref{trunc},
		\begin{eqnarray}\label{main2}
			\Vert P_{(A\cup G)^c}(f-P_A(f))+t 1_{\eta(G\setminus A)}\Vert=\Vert T_t(f-P_A(f))\Vert\leq \mathcal F^*\Vert f-P_A(f)\Vert.
		\end{eqnarray}
		Thus, by \eqref{main2} and \eqref{main1}, the basis is almost-greedy with constant $C_{al}\leq  (\mathcal F^*)^2$ for elements $f\in\mathbb X_{fin}$. Now, applying  Corollary \ref{bb3}, the results follows.
	\end{proof}
	
	\section{Properties (F$_p$) and (F$_p^*$)}\label{ff2}
	
	
	%
	%
	In all the results presented in Section \ref{ff} we can change democracy by conservativeness or super-conservativeness and Property (F) and (F$^*$) by Property (F$_p$) and (F$_p^*$) and obtain the same results. Here, we only present the fundamental theorem that is the version of Theorem \ref{p1} to study how the constants change.
	\begin{thm}\label{new}
		A basis $\mathcal B$ in a Banach space $\mathbb X$ has the Property (F$_p$) if and only if $\mathcal B$ is quasi-greedy and conservative. Moreover,
		$$\max\lbrace \Delta_c, C_q\rbrace\leq \mathcal F_p\leq 2+C_q+2C_q\Delta_c.$$
	\end{thm}
	\begin{proof}
		Assume that $\mathcal B$ has the Property (F$_p$) with constant $\mathcal F_p$. Taking $A=\emptyset$, we have that 
		\begin{eqnarray}
			\Vert f\Vert\leq C\Vert f+1_B+g\Vert,
		\end{eqnarray}
		for any $f, g$ and $B$ as in the definition of the Property (F$_p$). Now, taking $B=\emptyset$ and considering $f':=f-\mathcal G_m(f)$ and $y=\mathcal G_m(f)$, 
		\begin{eqnarray}
			\Vert f-\mathcal G_m(f)\Vert=\Vert f'\Vert\leq \mathcal F_p\Vert f'+g\Vert=\mathcal F_p\Vert f\Vert,
		\end{eqnarray}
		so the basis is quasi-greedy for elements with finite support. Applying now Corollary \ref{bb2}, the basis is quasi-greedy with $C_q\leq\mathcal F_p$. Now, on the other hand, taking $f=g=0$, we obtain conservativeness with constant $\Delta_{c}\leq \mathcal F_p$. 
		
		Now, take $f, g, A$ and $B$ as in the definition of Property (F$_p$). If we have $g=g_1+g_2$ where
		$$\text{supp}(g_1)=\lbrace n\in\text{supp}(g) : \vert e_n^*(g)\vert<1\rbrace,$$
		\begin{eqnarray*}
			\Vert f+1_A\Vert &\leq& \Vert f+g+1_B\Vert+\Vert g+1_B\Vert+\Vert 1_A\Vert\\
			&\leq&  2\Vert f+g+1_B\Vert+\Vert f \Vert+\Vert 1_A\Vert\\
			&\leq& (2+C_q)\Vert f+g+1_B\Vert+\Delta_c\Vert 1_B\Vert\\
			&\leq& (2+C_q)\Vert f+g+1_B\Vert+\Delta_c\Vert f+g_1+1_B\Vert+\Delta_c\Vert f+g_1\Vert\\
			&\leq& (2+C_q)\Vert f+g+1_B\Vert+2C_q\Delta_c\Vert f+g+1_B\Vert\\
			&=&(2+C_q+2C_q\Delta_c)\Vert f+g+1_B\Vert.
		\end{eqnarray*}
	\end{proof}
	%
	%
	
	\section{Proof of Theorem \ref{maintwo}}\label{ff3}

	%
	%
	%
	
	\begin{proof}[Proof of Theorem \ref{maintwo}]
		Assume now that $\mathcal{B}$ is $C_p$-partially-greedy and prove the Property (F$^*_p$). Take $f,z,y\in\mathbb X_{fin}$ satisfying from $i)$ to $iv)$ in the definition of Property (F$_p^*$). We write $y=1_{\eta D}+y_1+y_2$, where
		$$\text{supp}(y_1)=\lbrace n\in\text{supp}(y) : \vert e_n^*(y)\vert<1\rbrace,\;\text{supp}(y_2)=\lbrace n\in\text{supp}(y) : \vert e_n^*(y)\vert>1\rbrace,$$
		and $\eta\equiv\lbrace \text{sign}(e_n^*(y))\rbrace$. Consider $A=\text{supp}(z)$.
		
		If $A=\emptyset$, applying Theorem \ref{mp}, the basis is quasi-greedy with $C_q\leq C_p$ and we can conclude that $\Vert f\Vert\leq \Vert f+y\Vert$.
		
		Assume now that $A\neq\emptyset$ and consider $m=\text{max} A$ and define $B=\{1,\ldots,m\}\setminus A$. It is clear that $m=\vert A\cup B\vert \leq \vert B\cup D \vert$. Now, for any choice $\varepsilon\in\mathcal E_A$, define $h:=f+1_{\varepsilon A}+y_2+1_{\eta D}+1_B$.
		
		Since partially-greediness implies quasi-greediness with constant $C_p$ (see Theorem \ref{mp}), we have 
		\begin{eqnarray*}
			\Vert f+1_{\varepsilon A}\Vert &=& \Vert h-\mathcal G_m(h)\Vert\leq C_p \inf_{k\leq m}\Vert h-S_k(h)\Vert \leq C_p\Vert f+y_2+1_{\eta D}\Vert \\
			&\leq& C_p \left( \Vert f \Vert + \Vert y_2+1_{\eta D} \Vert\right).
		\end{eqnarray*}
		For the first element of the sum, consider $w:=f+y$ and we have  $$\Vert f \Vert=\Vert w-\mathcal G_n(w)\Vert,$$ with $n=\vert\text{supp}(y)\vert$. Then, applying quasi-greediness, we obtain $\Vert f\Vert\leq C_p\Vert f+y\Vert$. For the second one, we write $w=f+y_1+y_2+1_{\eta D}$ and using quasi-greediness, we have $$\Vert y_2+ 1_{\eta D}\Vert=\Vert \mathcal G_m(w)\Vert\leq(1+C_p)\Vert f+y\Vert,$$
		where $m=\vert \text{supp}(y_2)\cup D\vert$.
		
		Using both bounds, we obtain $\Vert f+1_{\varepsilon A}\Vert\leq C_p\left(1+2C_p\right)\Vert f+y\Vert$. Because of Lemma\ref{convex1}, we conclude that $\Vert f+z\Vert\leq C_p\left(1+2C_p\right)\Vert f+y\Vert$.

		Prove now b). Without loss of generality we can assume that $f\in\mathbb X_{fin}$ using Corollary \ref{bb4} and that $\Vert \tilde{f}\Vert_\infty\leq 1$. Start considering $A=\text{supp}(\mathcal G_m(f))$, $k\leq m$ and $B=\{1,\ldots,k\}$. If $A=B$, then the result is trivial. If $A\neq B$, we can decompose $$f-\mathcal{G}_m(f)=P_{(A\cup B)^c}(f-S_k(f))+P_{B\setminus A}(f).$$

		Let $f'=\frac{1}{t}P_{(A\cup B)^c}(f-S_k(f))$ and $z=\frac{1}{t}P_{B\setminus A}(f)$ with $t=\min_{n\in A}\vert e_n^*(f)\vert$ and $y=1_{\varepsilon (A\setminus B)}$ with $\varepsilon\equiv\lbrace\text{sign}(e_n^*(f)\rbrace$. Of course, $f'\cdot z=0$, $f'\cdot y=0$ and $y\cdot z=0$, $\Vert \tilde{f'}\Vert_\infty\leq 1$ since $\vert e_n^*(P_{(A\cup B)^c}(f-S_k(f)))\vert\leq t$ for $n\in (A\cup B)^c$ and $\vert A\setminus B\vert\geq \vert B\setminus A\vert$. Then, $f', z$ and $y$ verify the items of the Property (F$_p^*$), so
		\begin{eqnarray*}
			\Vert f-\mathcal G_m(f)\Vert = t\Vert f'+z\Vert\leq t\mathcal F_p^*\Vert f'+1_{\varepsilon(A\setminus B)}\Vert=\mathcal F_p^*\Vert P_{(A\cup B)^c}(f-S_k(f))+t1_{\varepsilon(A\setminus B)}\Vert.
		\end{eqnarray*}
		
		It turns out that $$P_{(A\cup B)^c}(f-S_k(f)) +t1_{\varepsilon A\setminus B}=T_t(f-S_k(f)),$$ where $T_t$ is the $t$-truncation operator. Now, since the Property (F$_p^*$) implies the Property (F$_p$) with the same constant, because of Theorem \ref{new} and Corollary \ref{bb2}, the basis is quasi-greedy with constant $C_q\leq\mathcal F_p^*$. Then, applying Lemma \ref{trunc}, we have that $\Vert T_t(f-S_k(f))\Vert\leq \mathcal F_p^*\Vert f-S_k(f)\Vert$. All together, we obtain $\Vert f-\mathcal{G}_m(f)\Vert\leq (\mathcal F_p^*)^2\Vert f-S_k(f)\Vert$ for all $k\leq m$ and hence, $\mathcal{B}$ is partially-greedy.
	\end{proof}
	
	\section{Annex}\label{ff4}
	In this annex, we write the main lemmas about density that we use in the paper.
	
	\begin{lem}{\cite[Lemma 7.2]{BB2}}\label{bb1}
		Let $\mathcal B$ be a basis for a Banach space $\mathbb X$. If $A$ is a greedy set for $f\in\mathbb X$, for every $\varepsilon>0$, there is $y\in\mathbb X_{fin}$ such that $\Vert f-y\Vert\leq \varepsilon$ and $A$ is a greedy set for $y$.
	\end{lem}
	
	\begin{cor}\label{bb2}
		Assume that $\mathcal B$ is a $C_q$-quasi-greedy basis of a Banach space $\mathbb X$ for elements with finite support. Then, $\mathcal B$ is quasi-greedy for every $f\in\mathbb X$.
	\end{cor}
	\begin{proof}
		Take $f\in\mathbb X$ and $A$ a greedy set of $f$ with order $m$. By Lemma \ref{bb1}, there is $y\in\mathbb X_{fin}$ such that $\Vert f-y\Vert\leq \varepsilon$ for every $\varepsilon>0$ with $A$ a greedy set for $y$. Then
		
		\begin{eqnarray*}
			\Vert f-P_A(f)\Vert &=& \Vert f-y-P_A(f)+y-P_A(y)+P_A(y)\Vert\\
			&\leq& \Vert f-y\Vert + \Vert y-P_A(y)\Vert + \Vert P_A(f-y)\Vert\\
			&\leq& \varepsilon(1+\Vert P_A\Vert)+C_q\Vert y\Vert\\
			&\leq& \varepsilon(1+\Vert P_A\Vert)+C_q\Vert f-y\Vert+C_q\Vert f\Vert\\
			&\leq& \varepsilon(1+C_q+\Vert P_A\Vert)+C_q\Vert f\Vert.
		\end{eqnarray*}
		Taking $\varepsilon\rightarrow 0$, we obtain the result.
	\end{proof}
	
	%
	
	\begin{cor}\label{bb3}
		Let $\mathcal B$ be a basis for a Banach space $\mathbb X$. If $\mathcal B$ is an almost-greedy basis for all $f\in\mathbb X_{fin}$ with constant $C_{al}$, then the basis is almost-greedy for every $f\in\mathbb X$ with the same constant.
	\end{cor}
	\begin{proof}
		Assume that $\mathcal B$ is almost-greedy for elements with finite support. Take $f\in\mathbb X$ with $A$ a greedy set of order $m$. Applying Lemma \ref{bb1}, for any $\varepsilon>0$, there is $g\in\mathbb X_{fin}$ such that $\Vert f-g\Vert\leq \varepsilon$  and $A$ a greedy set for $g$. Consider the set $B_1$ the set such that
		$$\inf_{\vert B\vert\leq m}\Vert f-P_B(f)\Vert=\Vert f-P_{B_1}(f)\Vert.$$
		
		\underline{Case 1:} $B_1=\emptyset$. 
		\begin{eqnarray*}
			\Vert f-P_A(f)\Vert &=& \Vert f-g+g-P_A(f)-P_A(g)+P_A(g)\Vert\\
			&\leq& \Vert f-g\Vert + \Vert g-P_A(g)\Vert + \Vert P_A(f-g)\Vert\\
			&\leq & \varepsilon(1 + \Vert P_A\Vert)+C_{al}\inf_{\vert B\vert\leq m}\Vert g-P_B(g)\Vert\\
			&\leq & \varepsilon(1 + \Vert P_A\Vert)+C_{al}\Vert g\Vert\\
			&\leq & \varepsilon(1 + \Vert P_A\Vert)+C_{al}\Vert f-g\Vert+C_{al}\Vert f\Vert\\	
			&\leq & \varepsilon(1+C_{al} + \Vert P_A\Vert)+C_{al}\Vert f\Vert.
		\end{eqnarray*}
		Taking $\varepsilon\rightarrow 0$, we obtain the result.

		\underline{Case 2:} $B_1\neq\emptyset$.
		\begin{eqnarray*}
			\Vert f-P_A(f)\Vert &=& \Vert f-g+g-P_A(f)-P_A(g)+P_A(g)\Vert\\
			&\leq& \Vert f-g\Vert + \Vert g-P_A(g)\Vert + \Vert P_A(f-g)\Vert\\
			&\leq & \varepsilon(1 + \Vert P_A\Vert)+C_{al}\inf_{\vert B\vert\leq m}\Vert g-P_B(g)\Vert\\
			&\leq & \varepsilon(1 + \Vert P_A\Vert)+C_{al}\Vert g-P_{B_1}(g)\Vert\\
			&\leq &  \varepsilon(1 + \Vert P_A\Vert) + C_{al}\Vert g-f+f-P_{B_1}(g)+P_{B_1}(f)-P_{B_1}(f)\Vert\\
			&\leq& \varepsilon(1 + \Vert P_A\Vert+C_{al})+C_{al}\Vert P_{B_1}(f-g)\Vert+\Vert f-P_{B_1}(f)\Vert\\
			&\leq& \varepsilon(1 + \Vert P_A\Vert+C_{al}+C_{al}\Vert P_{B_1}\Vert)+\Vert f-P_{B_1}(f)\Vert
		\end{eqnarray*}
		Taking $\varepsilon\rightarrow 0$, we obtain the result.
	\end{proof}
	
	With the same arguments, it is straightforward to show the next result.
	\begin{cor}\label{bb4}
		Let $\mathcal B$ be a basis for a Banach space $\mathbb X$. If $\mathcal B$ is a partially-greedy basis for all $f\in\mathbb X_{fin}$ with constant $C_{p}$, then the basis is partially-greedy for every $f\in\mathbb X$ with the same constant.
	\end{cor}
	
	\begin{lem}{\cite[Lemma 3.2]{BDKOW}}\label{sy}
		Let $\mathbb X$ be a Banach space. Suppose $D$ is a finite subset of $\mathbb N$, and $f\in\mathbb X\setminus\lbrace 0\rbrace$ satisfies $\text{supp}(f)\cap D=\emptyset$. Then, for any $\varepsilon>0$ there is $y\in\mathbb X_{fin}$ so that $\Vert f-y\Vert<\varepsilon$, $\text{supp}(y)\cap D=\emptyset$ and $\Vert\tilde{f}\Vert_{\infty}=\Vert\tilde{y}\Vert_\infty$.
	\end{lem}
	%
	
	\vspace{6pt}

\end{document}